\documentclass[12pt]{article}

\usepackage{graphicx}
\usepackage{amsmath}
\usepackage{url}

\newtheorem{theorem}{Theorem}[section]
\newtheorem{proposition}{Proposition}[section]
\newtheorem{corollary}[theorem]{Corollary}

\title{Bijections in de Bruijn Graphs}
\author{Josef Rukavicka\thanks{Department of Mathematics,
Faculty of Nuclear Sciences and Physical Engineering, CZECH TECHNICAL UNIVERSITY 
IN PRAGUE
(josef.rukavicka@seznam.cz).}}

\newenvironment{remark}[1][Remark]{\begin{trivlist}
\item[\hskip \labelsep {\bfseries #1}]}{\end{trivlist}}

\newenvironment{proof}[1][Proof]{\begin{trivlist}
\item[\hskip \labelsep {\bfseries #1}]}{\end{trivlist}}

\date{\small{February 02, 2016}
   } 
\begin{document}
\maketitle

\begin{abstract}
A T-net of order $m$ is a graph with $m$ nodes and $2m$ directed edges, where every node has indegree and outdegree equal to $2$. (A well known example of T-nets are de Bruijn graphs.)
Given a T-net $N$ of order $m$, there is the so called "doubling" process that creates a T-net $N^*$ from $N$ with $2m$ nodes and $4m$ edges. Let $\vert X\vert$ denote the number of Eulerian cycles in a graph $X$. It is known that $\vert N^*\vert=2^{m-1}\vert N\vert$. In this paper we present a new proof of this identity. Moreover we prove that $\vert N\vert\leq 2^{m-1}$.\\ Let $\Theta(X)$ denote the set of all Eulerian cycles in a graph $X$ and $S(n)$ the set of all binary sequences of length $n$. Exploiting the new proof we construct a bijection $\Theta(N)\times S(m-1)\rightarrow \Theta(N^*)$, which allows us to solve one of Stanley's open questions: we find a bijection between de Bruijn sequences of order $n$ and $S(2^{n-1})$.
\end{abstract}

\section{Introduction}

In 1894, A. de Rivi\`{e}re formulated a question about existence of circular arrangements of $2^n$ zeros and ones in such a way that every word of length $n$ appears exactly once, \cite{RIVIERE}. Let $B_0(n)$ denote the set of all such arrangements. 
(we apply the convention that the elements of $B_0(n)$ are binary sequences that start with $n$ zeros).
The question was solved in the same year by C. Flye Sainte-Marie, \cite{MARIE}, together with presenting a formula for counting these arrangements: $\vert B_0(n)\vert =2^{2^{n-1}-n}$. However the paper was then forgotten. The topic became well known through the paper of N.G. de Bruijn, who proved the same formula for the size of $B_0(n)$, \cite{DEBRUIJN}. Some time after, the paper of C. Flye Sainte-Marie was rediscovered by Stanley, and it turned out that both proofs were principally the same, \cite{DEBRUIJN2}. 

The proof uses a relation between $B_0(n)$ and the set of Eulerian cycles in a certain type of T-nets: A T-net $N$ of order $m$ is defined as a graph with $m$ nodes and $2m$ directed edges, where every node has indegree and outdegree equal to $2$ (a T-net is often referred as a balanced digraph with indegree and outdegree of nodes equal to $2$, see for example \cite{STANLEY_AC}). N.G. de Bruijn defined a doubled T-net $N^*$ of $N$. A doubled T-net $N^*$ of $N$ is a T-net such that:
\begin{itemize}
\item each node of $N^*$ corresponds to an edge of $N$
\item two nodes in $N^*$ are connected by an edge if their corresponding edges in $N$ are incident and the ending node of one edge is the starting node of the second edge.
\end{itemize}
\begin{remark}
We call two edges to be incident if they share at least one common node; the orientation of edges does not matter.
\end{remark}
As a result $N^*$ has $2m$ nodes and $4m$ edges, see an example on Figure \ref{fg_doubling_simple}. (A doubled T-net of $N$ is known as well as a line graph of $N$, \cite{KISHORE}.)

\begin{figure}
\centering
\includegraphics[width=2.0in]{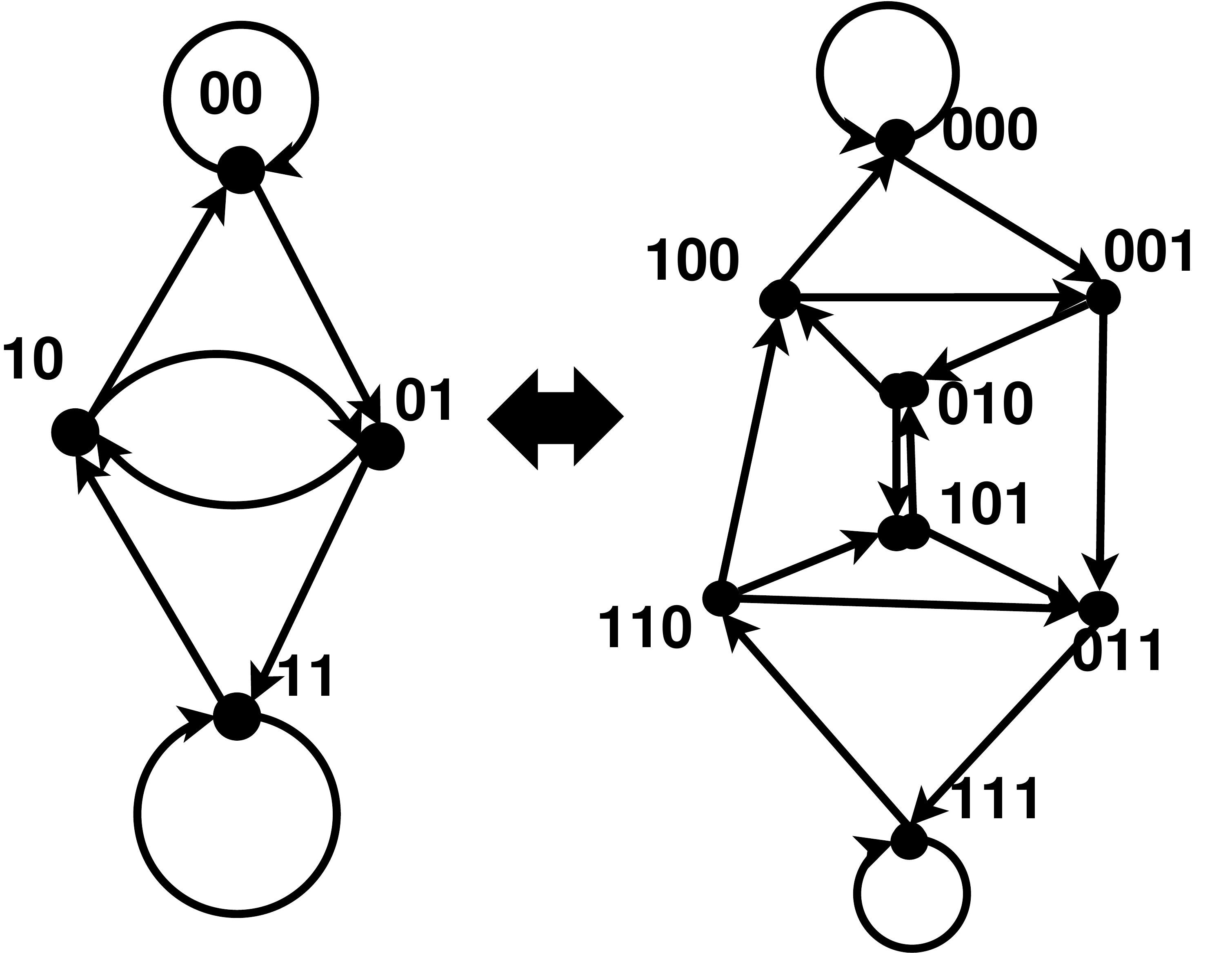} 
\caption{A doubling of a de Bruijn graph: $N$ and $N^*$}
\label{fg_doubling_simple}
\end{figure}

Let $\Theta(X)$ be the set of all Eulerian cycles in $X$ and let $\vert X \vert = \vert \Theta(X)\vert $ denote the number of Eulerian cycles in $X$, where $X$ is a graph. It was proved inductively that $\vert N^*\vert=2^{m-1}\vert N\vert$. Moreover N.G. de Bruijn constructed a T-net (nowadays called a "de Bruijn graph") whose Eulerian cycles are in bijection with the elements of $B_0(n)$. 

A de Bruijn graph $H_n$ of order $n$ is a T-net of order $2^n$, whose nodes correspond to the binary words of length $n-1$. A node $s_1s_2\dots s_{n-1}$ has two outgoing edges to the nodes $s_2\dots s_{n-1}0$ and $s_2\dots s_{n-1}1$. It follows that a node $s_1s_2\dots s_{n-1}$ has two incoming edges from nodes $0s_1s_2\dots s_{n-2}$ and $1s_1s_2\dots s_{n-2}$. Given an edge $e$ going from the node $s_1s_2\dots s_{n-1}$ to the node $s_2\dots s_{n-1}s_n$, then the edge $e$ corresponds to the  word $s_1s_2\dots s_{n-1}s_n$ of length $n$, which implies the natural bijection between Eulerian cycles $\Theta(H_n)$ and binary sequences $B_0(n)$, \cite{DEBRUIJN}. That is why we will write $B_0(n)\equiv \Theta(H_n)$.

De Bruijn graphs found several interesting applications, among others in networking, \cite{BAKER}, and bioinformatics, \cite{PAVEL}, \cite{ZERBINO}.

The important property of de Bruijn graphs is that a doubled T-net of a de Bruijn graph of order $n$ is a de Bruijn graph of order $n+1$, see an example on Figure \ref{fg_doubling_simple} of the de Bruijn graph of order $3$ ($H_3=N$) and of order $4$ ($H_4=N^*)$. Since $\vert B_0(2)\vert=1$ ($B_0(2)=\{0011\}$) it has been derived that $\vert B_0(n)\vert =2^{2^{n-1}-n}$, \cite{BAKER}, \cite{DEBRUIJN}, \cite{DEBRUIJN2}. 

There is also another proof using matrix representation of graphs, \cite{STANLEY_AC}. Yet it was an open question of Stanley, \cite{STANLEY_OP}, \cite{STANLEY_AC}, if there was a bijective proof:
\begin{quote}
Let $B(n)$ be the set of all binary de Bruijn sequences of
order $n$, and let $S(n)$ be the set of all binary sequences of length $n$. Find an explicit bijection $B(n) \times  B(n)\rightarrow S(2^n)$.
\end{quote}
This open question was solved in 2009, \cite{KISHORE}, \cite{STANLEY_AC}. 

\begin{remark}
In the open question of Stanley, $B(n)$ denotes the de Bruijn sequences that do not necessarily start with $n$ zeros like in the case of $B_0$. $B(n)$ contains all $2^n$ "circular rotations" of all sequences from $B_0(n)$; formally, given $s=s_1s_2\dots s_{2^n}\in B_0(n) $, then $s_is_{i+1}\dots s_{2^n}s_1s_2\dots s_{i-1}\in B(n)$, where $1\leq i \leq 2^n$. It is easy to see that all these $2^n$ "circular rotations" are distinct binary sequences. It follows that $\vert B(n)\vert=2^n\vert B_0(n)\vert$. Hence it is enough to find a bijection $B_0(n)\rightarrow S(2^{n-1}-n)$ to solve this open question.
\end{remark}

In this paper we present a new proof of the identity $\vert N^*\vert=2^{m-1}\vert N\vert$, which allows us to prove that $\vert N\vert\leq 2^{m-1}$ and to construct a bijection $\nu : \Theta(N)\times S(m-1)\rightarrow \Theta(N^*)$ and consequently to present another solution to the Stanley's open question: We define $\rho_2(\epsilon)=0011$ (recall that $B_0(2)=\{0011\}$) and 
let $\rho_{n} : S(2^{n-1}-n) \rightarrow B_0(n)$ be a map defined as $\rho_n(s)=\nu(\rho_{n-1}(\dot s),\ddot s)$, where $\epsilon$ is the binary sequence of length $0$,
 $n>2$, $s=\dot s\ddot s$, $\dot s \in S(2^{n-2}-(n-1))$, and  $\ddot s \in S(2^{n-2}-1)$. 
\begin{proposition}
The map $\rho_n$ is a bijection.
\end{proposition}
\begin{proof}
Note that $\dot s \in S(2^{n-2}-(n-1))$ and $\vert B_0(n-1)\vert=2^{(n-1)-1}-(n-1)=2^{n-2}-(n-1)$; thus $\dot s$ is a valid input for the function $\rho_{n-1}$ and $\rho_{n-1}(\dot s)\in B_0(n-1)\equiv \Theta(H_{n-1})$. In addition, $H_{n-1}$ has $m=2^{n-2}$ nodes and $\ddot s\in S(2^{n-2}-1)$ has the length $m-1$, hence it makes sense to define $\rho_n(s)=\nu(\rho_{n-1}(\dot s),\ddot s)$. Because $\nu$ is a bijection, see Proposition \ref{pr_bij_doublig_tnet}, it is easy to see by induction on $n$ that $\rho_n$ is a bijection as well.
\end{proof}
\begin{remark}
Less formally said, the bijection $\rho_n(s)$ splits the binary sequence $s$ into two subsequences $\dot s$ and $\ddot s$. Then the bijection $\rho_{n-1}$ is applied to $\dot s$, the result of which is a de Bruijn sequence $p$ from $B_0(n-1)$ (and thus an Eulerian cycle in $H_{n-1}$). Then the bijection $\nu$ is applied to $p$ and $\ddot s$. The result is a de Bruijn sequence from $B_0(n)$.
\end{remark}

\section{A double and quadruple of a T-net}

\noindent
Let $Y$ be a set of graphs; we define 
$\Theta(Y)=\bigcup_{X\in Y}\Theta(X)$ (the union of sets of Eulerian cycles in graphs from $Y$) and $\vert Y\vert =\sum_{X\in Y}\vert X\vert$ (the sum of the numbers of Eulerian cycles). Let $U(X)$ denote the set of nodes of a graph $X$.

\begin{figure}
\centering
\includegraphics[width=2.0in]{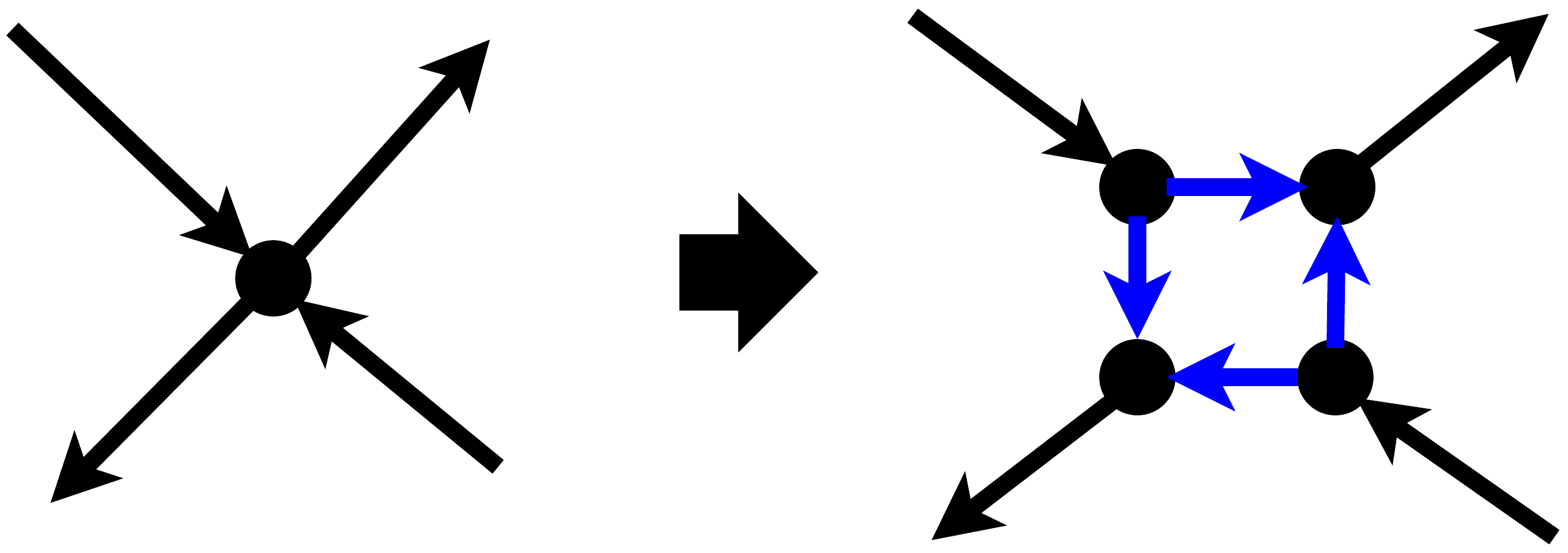} 
\caption{A node replacing by $4$ nodes and $4$ edges}
\label{fg_node_replace}
\end{figure} 

We present a new way of constructing a doubled T-net, which will enable us to show a new non-inductive proof of the identity $\vert N^{*}\vert=2^{m-1}\vert N\vert$ and to prove $\vert N\vert\leq 2^{m-1}$. 
\begin{figure}
\centering
\includegraphics[width=2.0in]{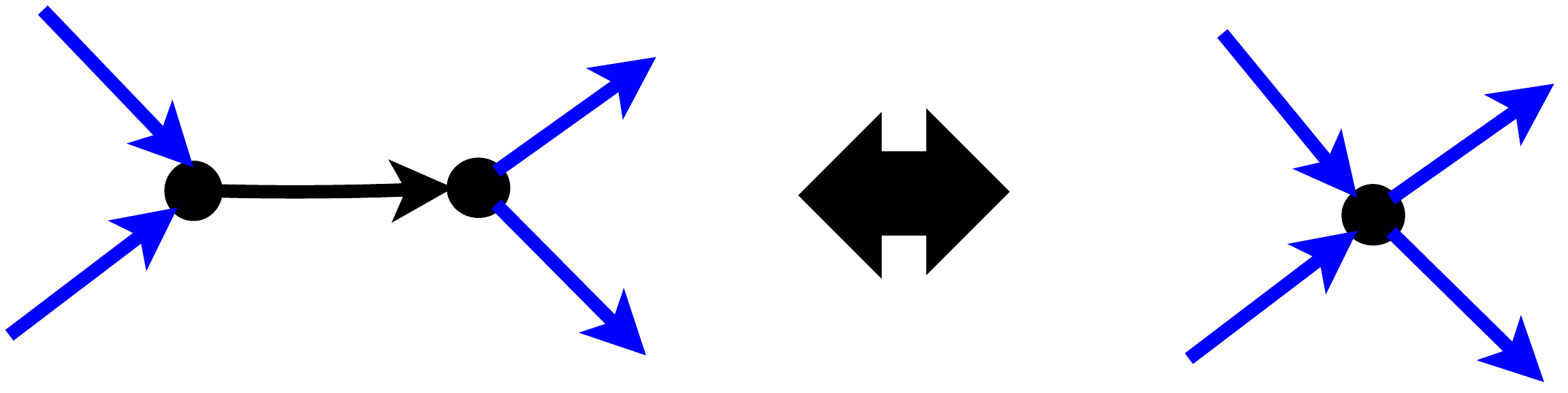} 
\caption{A removing black edges and fusion of nodes}
\label{fg_node_fusion}
\end{figure}

We introduce a quadruple of $N$  denoted by $\hat N$: The quadruple $\hat N$ arises from $N$ by replacing every node $a\in U(N)$ by 4 nodes and 4 edges as depicted on the Figure \ref{fg_node_replace}. Let $\Gamma(a)$ denote the set of these 4 nodes and $\Pi(a)$ denote the set of these 4 edges that have replaced the node $a$. The edges from $\Pi(a)$ are in blue color on the figures and we will distinguish blue and black edges as follows: In a graph containing at least one blue edge, we define an Eulerian cycle to be a cycle that traverses all blue edges exactly once and all black edges exactly twice, see Figure \ref{fg_doubling}. 

\begin{figure}
\centering
\includegraphics[width=3.0in]{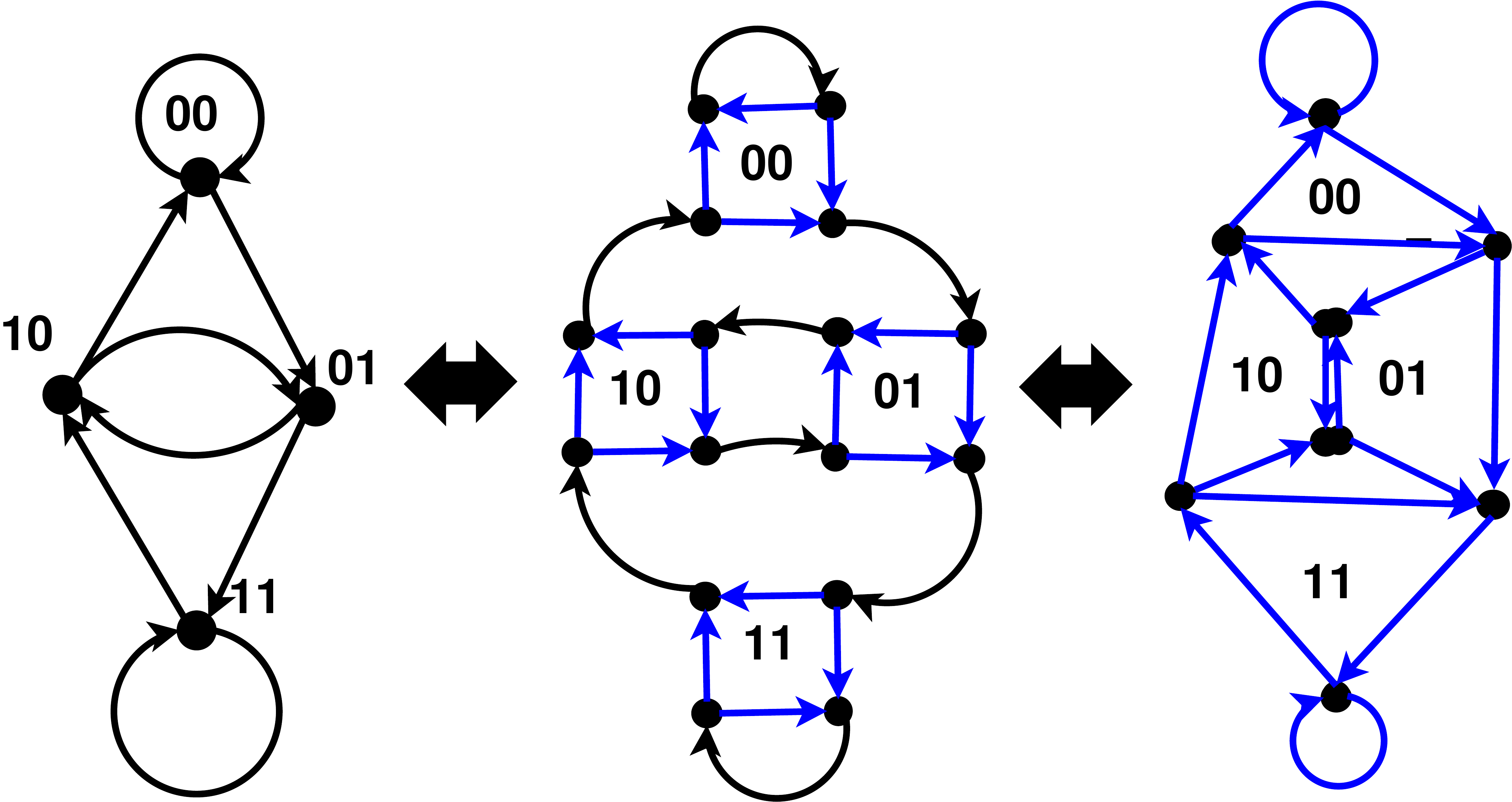} 
\caption{An example of $N$, $\hat N$, and $N^*$}
\label{fg_doubling}
\end{figure}

\begin{remark}
Note that a quadruple $\hat N$ is not a T-net, since the indegree and outdegree are not always equal to $2$. But since  the black edges can be traversed twice, we can consider them as parallel edges (two edges that are incident to the same two nodes). Then it would be possible to regard $\hat N$ as a T-net.
\end{remark}

By removing black edges and "fusing" their incident nodes into one node in $\hat N$ (as depicted on Figure \ref{fg_node_fusion}), we obtain a doubled T-net $N^{*}$ of $N$. And the reverse process yields $\hat N$ from $N^*$: turn all edges from black to blue and then replace every node by two nodes connected by one black edge, where one node has two outgoing blue edges and one incoming black edge and the second node two incoming edges and one outgoing black edge. Thus we have a natural bijection between Eulerian cycles in $\hat N$ and $N^*$. See an example on Figure \ref{fg_doubling}. 
\begin{remark}
If all edges in a graph are in one color, then it makes no difference if they are black or blue. An Eulerian cycle traverses in that case just once every edge.
\end{remark}

\begin{figure}
\centering
\includegraphics[width=3.0in]{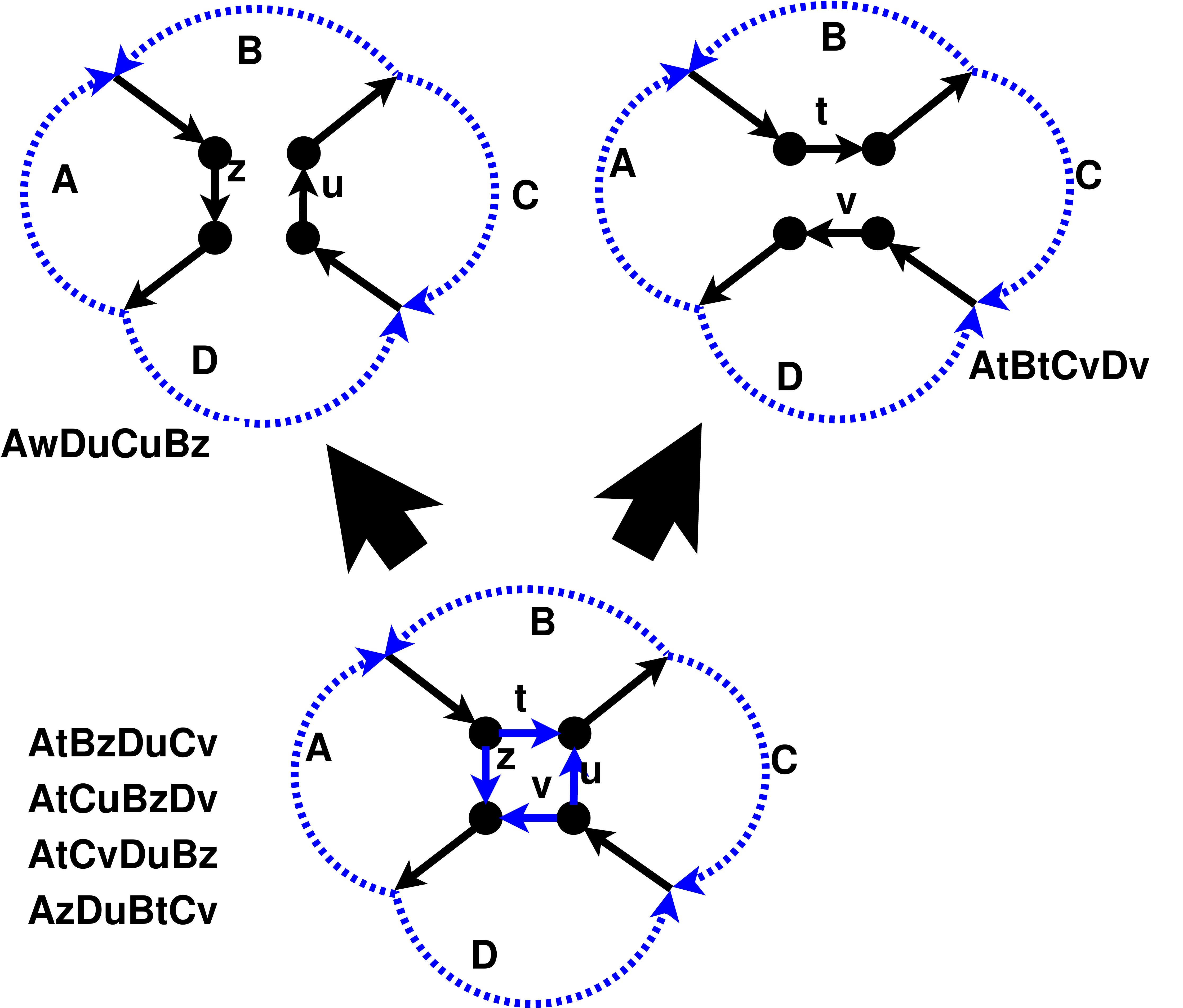} 
\caption{Edges replacement. Case I}
\label{fg_edges_repl_a}
\end{figure}
\begin{figure}
\centering
\includegraphics[width=4.0in]{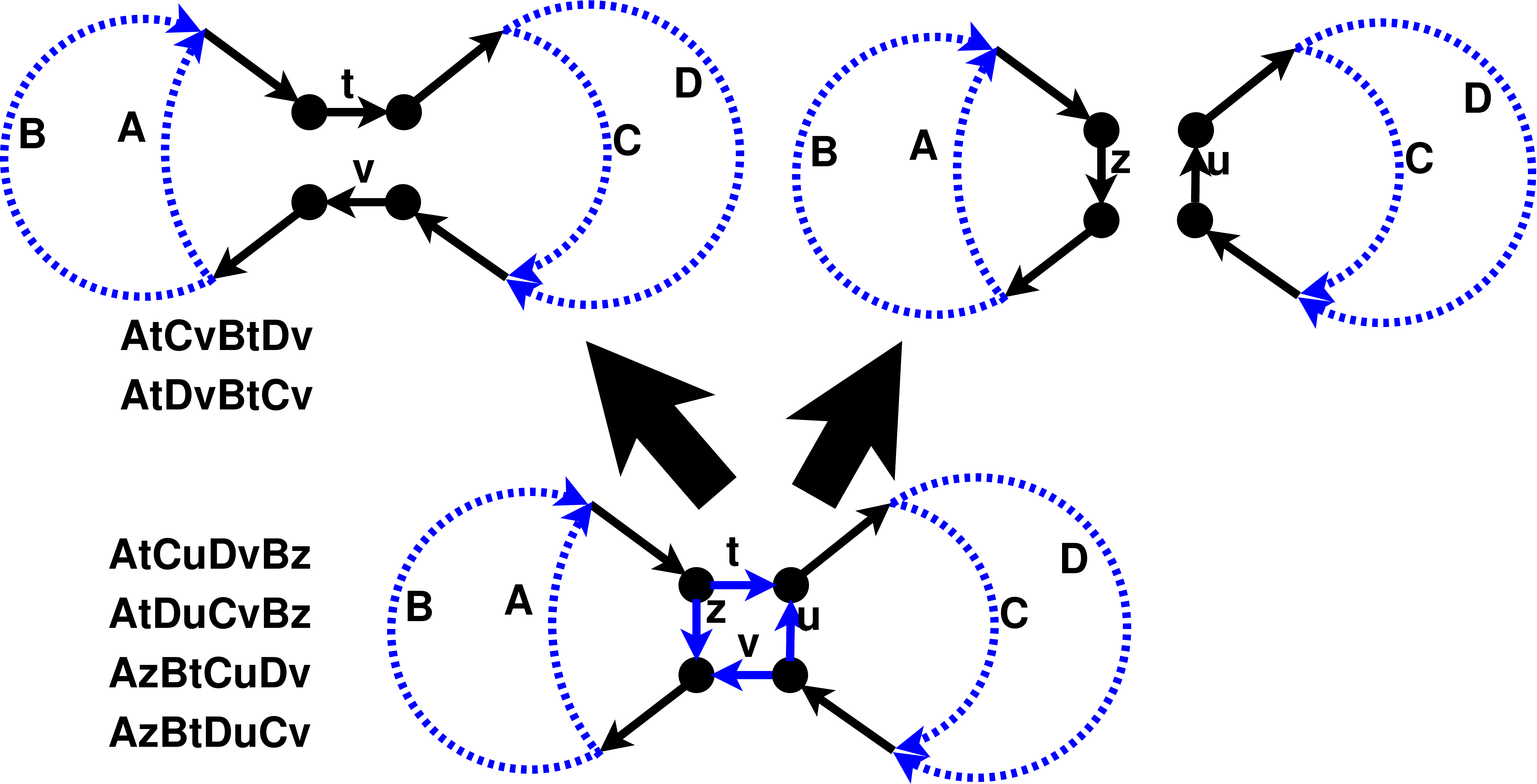} 
\caption{Edges replacement. Case II}
\label{fg_edges_repl_b}
\end{figure}

Fix an order on nodes $U(N)$. As a result we have a bijection $\phi: \{1,2,\dots ,m\} \rightarrow U(N)$.
Given $i\in \{1,2,\dots ,m\}$, let us denote the edges from $\Pi(\phi(i))$ by $t,u,v,z$, in such a way that $t$ and $v$ are not incident edges; it follows that $u$ and $z$ are not incident as well. 

Let $W_0=\{\hat N\}$, we define $W_i=\{\dot w,\ddot w\mid w\in W_{i-1}\}$, where $i\in\{1,2,\dots ,m\}$  and $\dot w$, $\ddot w$ are defined as follows: 
We construct the graph $\dot w$ by removing edges $t,v$ from $w$ and by changing the color of $u,z$ from blue to black (thus allowing the edges $u,z$ to be traversed twice). Similarly we construct $\ddot w$ from $w$ by removing edges $u,z$ and by changing the color of $t,v$ from blue to black, where $t,u,v,z\in \Pi(\phi(i))$. 

The crucial observation is:
\begin{proposition}
\label{pr_half_of_euler_paths}
Let $w\in W_i$, where $i\in\{0,1,\dots ,m-2\}$. Then 
$\vert w \vert= 2\vert \dot w\vert + 2\vert \ddot w\vert$.
\end{proposition}
\begin{remark}
The following proof is almost identical to the one in \cite{DEBRUIJN}, where the author constructed two graphs $d_1, d_2$ from a graph $d$ and proved that $\vert d\vert = 2\vert d_1\vert + 2\vert d_1\vert$
\end{remark}
\begin{proof} 
Given an Eulerian cycle $g$ in $w$, then split $g$ in four paths $A,B,C,D$ and edges $t,u,v,z \in \Pi(\phi(i))$. We will count the number of Eulerian cycles in $\dot w, \ddot w$ that are composed from all 4 paths $A,B,C,B$ and that differ only in their connections on edges $t,u,v,z$.
Exploiting the N.G. de Bruijn's notation, all possible cases are depicted on Figures \ref{fg_edges_repl_a} and \ref{fg_edges_repl_b}. 
\begin{itemize}
\item In case I, the graph $w$ contains $4$ Eulerian cycles: AtBzDuCv, AtCuBzDv, AtCvDuBz, AzDuBtCv; whereas the graphs $\dot w$ and $\ddot w$ have together $2$ Eulerian cycles: AzDuCuBz and AtBtCvDv. Thus $\vert w\vert=4$ and $\vert \dot w\vert$+$\vert \ddot w\vert=2$.
\item In case II, the graph $w$ contains $4$ Eulerian cycles: AtCuDvBz, AtDuCvBz, AzBtCuDv, AzBtDuCv; whereas the graph $\ddot w$ has $2$ Eulerian cycles: AtCvBtDv, AtDvBtCv. The graph $\dot w$ is disconnected and therefore $\dot w$ has $0$ Eulerian cycles. Thus $\vert w\vert=4$ and $\vert \dot w\vert$+$\vert \ddot w\vert=2$.
In case II, it is possible the $A=B$ or $C=D$. In such a case, $\vert w\vert=2$ and $\vert \dot w\vert$+$\vert \ddot w\vert=1$.
\end{itemize}
This ends the proof.
\end{proof}

\noindent
We define $\Delta = \{w \mid w\in W_m \mbox{ and $w$ is connected}\}$. The Figure \ref{fg_tree} shows an example of all iterations and construction of graphs in $\Delta$ from the graph $\hat N$, where $N$ is a de Bruijn graph of order $3$. The order of nodes from $N$ is $00<10<01<11$. Most of the disconnected graphs are ommited.
 
\begin{figure}
\centering
\includegraphics[width=4.0in]{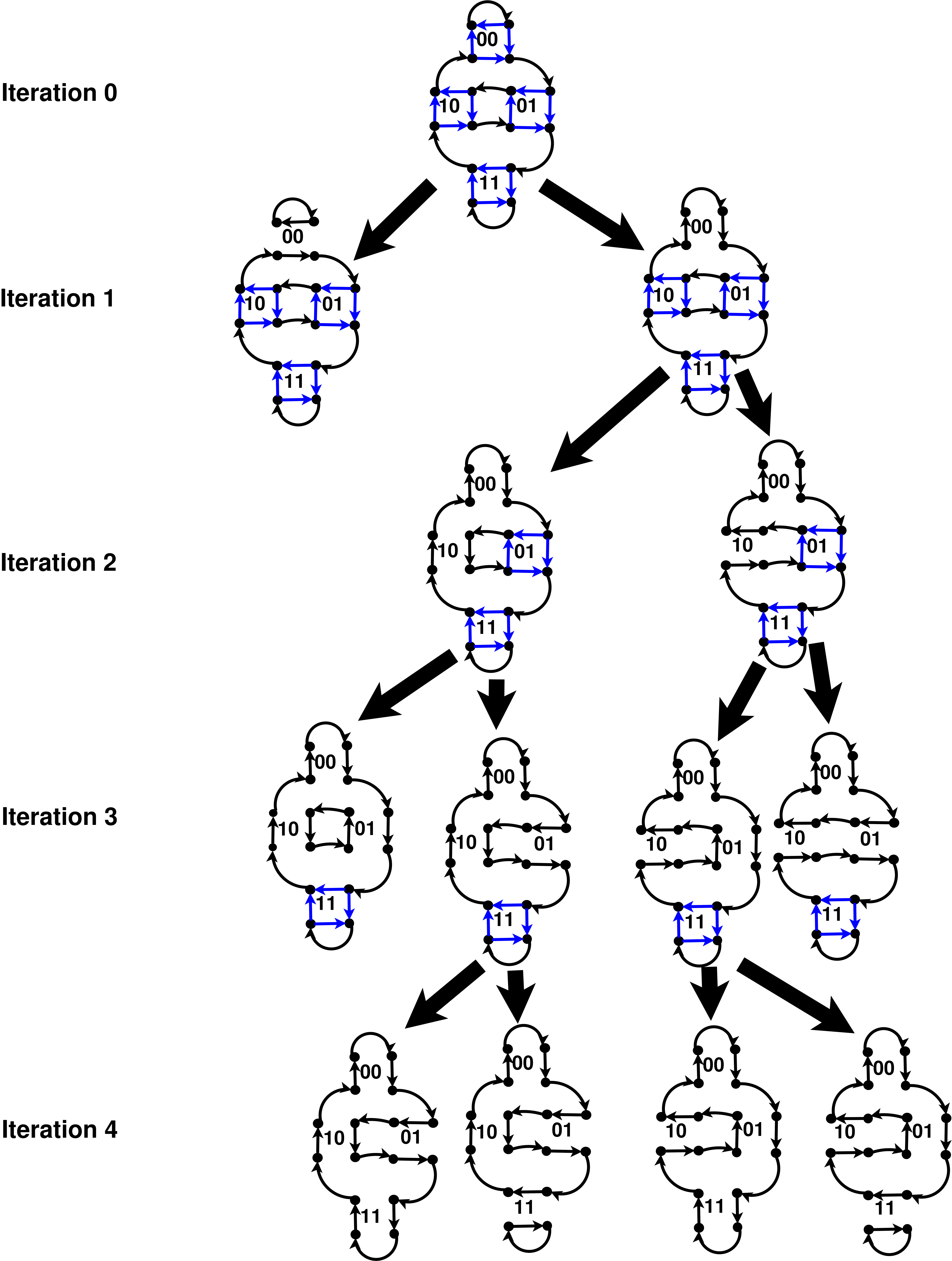} 
\caption{Constructing the set $\Delta$ from $\hat N$}
\label{fg_tree}
\end{figure}

\begin{remark}
In the previous proof in case II, it can happen that $A=B$ or $C=D$. Note in the iteration step $i=m$ (when constructing $W_m$ from $W_{m-1}$) it holds that $A=B$ and $C=D$, because all nodes have indegree and outdegree equal to $1$ with exception of nodes $\Gamma(\phi(m-1))$. Hence $\vert W_{m-1}\vert = \vert W_{m}\vert$. It follows as well that every connected graph $w\in W_{m-1}$ has exactly one Eulerian cycle. That is why in the Proposition \ref{pr_half_of_euler_paths} we consider $i\in\{0,1,\dots ,m-2\}$.
\end{remark}

\begin{corollary}
\label{cor_rel_iterated_sets}
$2\vert W_{i-1}\vert = \vert W_i\vert$ and $\vert W_{m-1}\vert = \vert W_m\vert$, where $i\in \{1,2,\dots ,m-1\}$.
\end{corollary}

\begin{proposition}
$2^{m-1}\vert \Delta \vert=\vert N^{*}\vert = \vert \hat N \vert$.
\end{proposition}
\begin{proof}
The only graphs in $W_m$ that contain an Eulerian cycle are connected graphs, it means only graphs from $\Delta$. On the other hand every graph $w\in \Delta$ contains exactly one Eulerian cycle, since every node has indegree and outdegree equal to $1$.
The proposition follows then from Corollary \ref{cor_rel_iterated_sets}, because $\vert \hat N\vert=\vert W_0\vert$ (recall that $W_0=\{\hat N\}$).
\end{proof}

\begin{proposition}
\label{pr_bij_n_delta}
There is a bijection between $\Theta(N)$ and $\Theta(\Delta)$ and $\Theta(W_{m-1})$ and $\Theta(W_{m})$.
\end{proposition}
\begin{proof}
Given a connected graph $w\in W_{m-1}$, then just one graph of $\dot w$ and $\ddot w$ is connected. Let us say it is $\dot w$. Recall that there is exactly one Eulerian cycle $AtCuCvAz$ in $w$ ($A=B$ and $C=D$, see Figure \ref{fg_edges_repl_b}). Then $AtCv$ is the only Eulerian cycle in $\dot w\in \Delta \subset W_{m}$. This shows a bijection between $\Theta(W_{m-1})$ and $\Theta(W_{m})$ and $\Theta(\Delta)$.

Let $\bar p = p_1p_2\dots p_{4m}$ be the only Eulerian cycle in $w \in \Delta$, where $p_i$ are edges of $w$. Without loss of generality suppose that $p_1\in \Pi(a)$ for some $a\in U(N)$ (it means that $p_1$ is a blue edge in $\hat N$). It follows that all $p_i$ with $i$ odd are blue edges in $\hat N$ all $p_i$ with $i$ even are edges from $N$ (they are black edges in $\hat N$); in consequence the path $p=p_2p_4\dots p_{4m}$ is an Eulerian cycle in $N$. A turning the Eulerian cycle in $w$ into the Eulerian cycle $p$ in $N$ is schematically depicted on Figure \ref{fg_node_replacement_to_simple}. Thus we have a bijection between $\Theta(N)$ and $\Theta(\Delta$). 
This ends the proof.
\end{proof}

\begin{figure}
\centering
\includegraphics[width=2.0in]{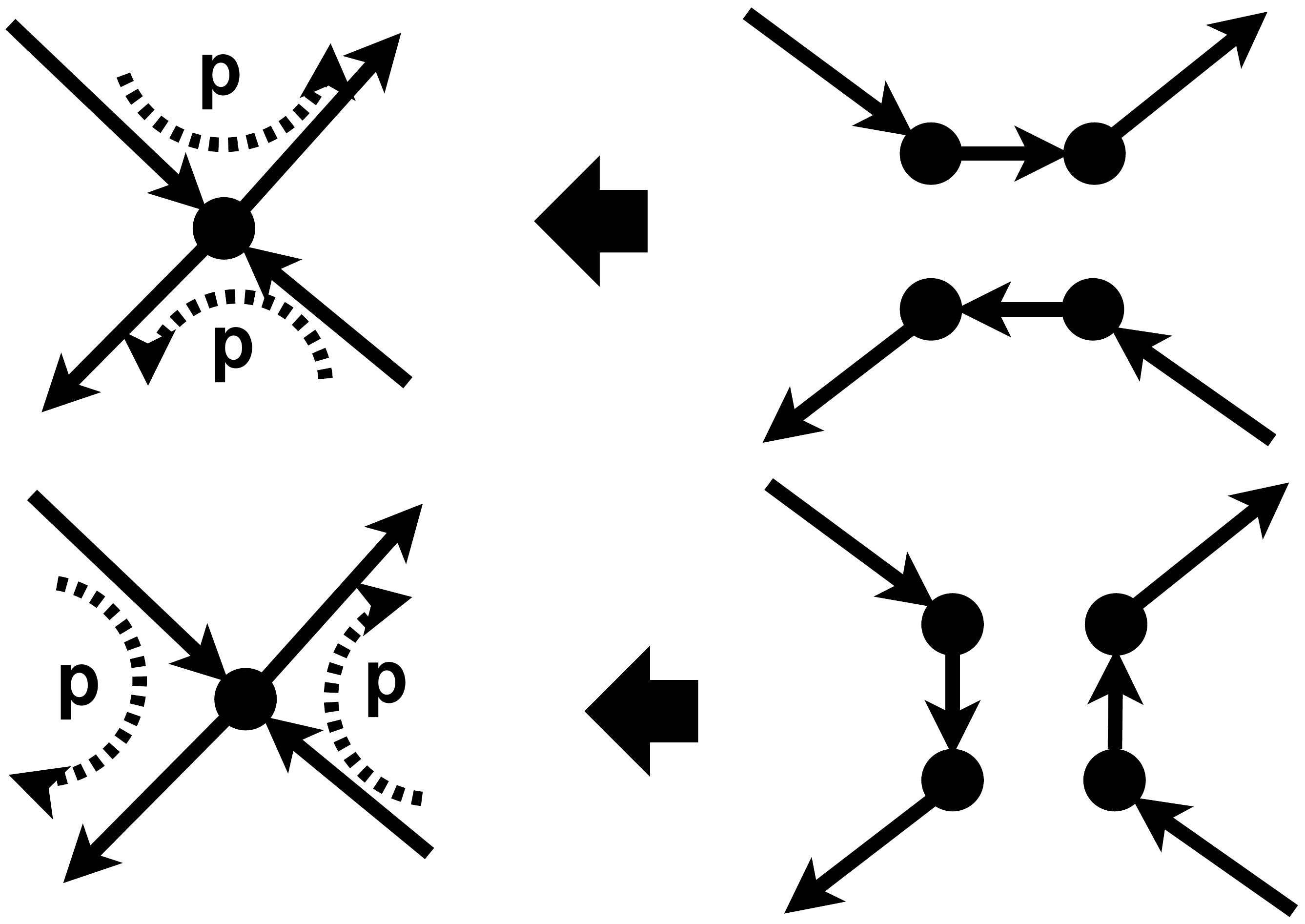} 
\caption{Converting a an Eulerian cycle from $\Delta$ into an Eulerian cycle in $N$}
\label{fg_node_replacement_to_simple}
\end{figure}

\begin{corollary}
Let $N$ be a T-net of order $m$. Then $\vert N\vert\leq 2^{m-1}$ Eulerian cycles.
\end{corollary}
\begin{proof}
The set $W_{m-1}$ contains $2^{m-1}$ graphs and recall that every connected graph $w\in W_{m-1}$ has exactly one Eulerian cycle. The result follows then from $\vert W_{m-1}\vert=\vert W_{m}\vert$ and $\Delta\subseteq W_{m}$.
\end{proof}

\section{Bijection of binary sequences and de Bruijn sequences}

Given $i\in \{1,2,\dots ,m\}$, in the previous section we agreed the edges from $\Pi(\phi(i))$ are denoted by $t,u,v,z$, in such a way that $t$ and $v$ are not incident edges (and consequently that $u$ and $z$ are not incident as well). For this section we need that these edges are ordered, hence let us suppose that it holds $t<u<v<z$. This will allow us to identify "uniquely" the edges.

Let us look again on the Figure \ref{fg_edges_repl_a}.
We can identify the path $A$ as the path between incident nodes of the edge $z$ that do not contain edges $t,u,v$. In a similar way we can identify $B,C,D$. 

On the Figure \ref{fg_edges_repl_b} we can not distinguish $A$ from $B$ and $C$ from $D$ only by edges $t,u,v,z$. If $A\not =B$, then let $\delta$ be the first node where $A$ and $B$ differ. The node $\delta$ has two outgoing blue edges, let us say they are $t,z$. We use this difference to distinguish $A$ and $B$. Let us define $A$ to be the path that follows the edge $t$ from $\delta$ and $B$ the path that follows the edge $z$ from $\delta$. Again in a similarly way we can distinguish $C$ from $D$. Hence let us suppose we have an "algorithm" that splits an Eulerian cycle $p\in \Theta(W_i)$  into the paths $A,B,C,D$ and edges $t,u,v,z \in \Pi(\phi(i))$ for given $N,i$ (recall that the nodes of $N$ are ordered and thus $i$ determines the node $\phi(i)\in U(N)$). We introduce the function $\omega_{N,i}: (p,\alpha)\rightarrow \Theta(W_{i-1})$, where
\begin{itemize}
\item
$N$ is a T-net of order $m$
\item
$i\in \{1,\dots,m-1\}$ 
\item
$p\in \Theta(W_i)$
\item
$\alpha\in \{0,1\}$
\end{itemize}

\begin{remark}
Less formally said, the function $\omega$ transform an Eulerian cycle $p\in \Theta(W_{i})$ into an Eulerian cycle $\bar p\in \Theta(W_{i-1})$ for given $N,i,\alpha$.
\end{remark}

\noindent
Given $N$ and $i$, we define for the case I (Figure \ref{fg_edges_repl_a}):\\
$\omega_{N,i}(AzDuCuBz,0)=AtBzDuCv$\\
$\omega_{N,i}(AzDuCuBz,1)=AtCuBzDv$\\
$\omega_{N,i}(AtBtCvDv,0)=AtCvDuBz$\\
$\omega_{N,i}(AtBtCvDv,1)=AzDuBtCv$\\
For the case II (Figure \ref{fg_edges_repl_b}), where $A\not=B$ and $C\not =D$:\\
$\omega_{N,i}(AtCvBtDv,0)=AtCuDvBz$\\
$\omega_{N,i}(AtCvBtDv,1)=AzBtCuDv$\\
$\omega_{N,i}(AtDvBtCv,0)=AtDuCvBz$\\
$\omega_{N,i}(AtDvBtCv,1)=AzBtDuCv$\\
For the case II where $A=B$ and $C\not =D$:\\
$\omega_{N,i}(AtCvAtDv,0)=AtCuDvAz$\\
$\omega_{N,i}(AtCvAtDv,1)=AtDuCvAz$\\
For the case II where $A\not =B$ and $C=D$:\\
$\omega_{N,i}(AtCvBtCv,0)=AtCuCvBz$\\
$\omega_{N,i}(AtCvBtCv,1)=AzBtCuCv$\\
Now, when we fixed an order on edges at the beginning of this section, it is necessary to distinguish another possibility in the case II, namely the paths $A,B$ can be paths between incident nodes of the edge $t$ that do not contain edges $u,v,z$ and $C,D$ can be paths between incident nodes of the edge $v$ that do not contain edges $t,u,z$. Obviously, in this case it is possible to define $\omega$ in a similar way. To save some space we do not present an explicit definition.

\begin{remark}
The previous definition of $\omega_{N,i}(p,\alpha)$ can be modified with regard to the reader's needs, including the way of recognition of paths $A,B,C,D$. It matters only that $\omega_{N,i}$ is injective. Our definition is just one possible way.
\end{remark}

\begin{remark}
To understand correctly the definition of $\omega$, recall that when comparing two Eulerian cycles, it does not matter which edge is written as the first one. For example 
the paths $AtCuDvAz$ and $AzAtCuDv$ are an identical Eulerian cycle. 
\end{remark}

Let $S(n)$ denote the set of all binary sequences of length $n$.

\begin{proposition}
\label{pr_bij_doublig_tnet}
Let $N$ be a T-net of order $m$, $s=s_1s_2\dots s_{m-1} \in S(m-1)$ be a binary sequence, and $p\in \Theta(N)$.
We define $p=p^{m-1}$ and $p^{i-1}=\omega_{N,i}(p^{i},s_i)$, where $i \in \{1,2,\dots ,m-1\}$.
Then the map $\nu : \Theta(N)\times S(m-1)\rightarrow \Theta(N^*)$ defined as $\nu(p,s)=p^{0}$ is a bijection.
\end{proposition}
\begin{proof}
Recall that there is a bijection between $\Theta(N)$ and $\Theta(W_{m-1})$, see Proposition \ref{pr_bij_n_delta}; hence we can suppose that $p\in W_{m-1}$.\\
The definition of the function $\omega$ implies that $\omega_{N,i}(p, \alpha)=\omega_{N,i}(\bar p, \bar \alpha)$ if and only if $p=\bar p$ and $\alpha=\bar \alpha$. It follows that $\nu$ is injective. In addition we proved that $\vert N\vert=\vert W_{m-1}\vert$ and that $2^{m-1}\vert N\vert=\vert \hat N\vert=\vert W_0\vert$. In consequence $\nu$ is surjective and thus bijective.
\end{proof}


\begin{thebibliography}{widest-label}

\bibitem{BAKER} Baker Joel: \emph{De Bruijn graphs and their applications to fault tolerant networks}, 2011, \url{http://hdl.handle.net/1957/26215}

\bibitem{DEBRUIJN} De Bruijn N. G.: \emph{A combinatorial problem}, Nederl. Akad. Wetensch., Proc., 49:758-764, 1946.

\bibitem{DEBRUIJN2} De Bruijn N. G.: \emph{Acknowledgement of priority to C. Flye Sainte-Marie on the counting of circular arrangements of 2n zeros and ones that show each n-letter word exactly once}, TH-Report 75-WSK-06, Technological University Eindhoven, the Netherlands, pp. 1–14, 1975

\bibitem{KISHORE} Bidkhori Hoda, Kishore Shaunak: \emph{Counting the spanning trees of a directed line graph}, 2009, \url{http://arxiv.org/pdf/0910.3442.pdf}

\bibitem{MARIE} C. Flye Sainte-Marie: \emph{Solution to question nr. 48}, l'Interm\'{e}diaire des Mathr\'{e}maticiens, 1894, 107-110.

\bibitem{PAVEL} Pevzner Pavel A., Tang Haixu, Waterman Michael S.:\emph{An Eulerian path approach to DNA fragment assembly}, Proceedings of the National Academy of Science, vol. 98, Issue 17, p.9748-9753, 2001, \url{https://dx.doi.org/10.1073/pnas.171285098}

\bibitem{RIVIERE} A. de Rivi\`{e}re: \emph{Question nr. 48}, l'Interm\'{e}diaire des Mathr\'{e}maticiens, 1894, 19-20.

\bibitem{ZERBINO} Zerbino Daniel R., Birney Ewan: \emph{Velvet: Algorithms for de novo short read assembly using de Bruijn graphs}, Genome Research 18 (5): 821–829, 2008, \url{https://dx.doi.org/10.1101/gr.074492.107}

\bibitem{STANLEY_OP} R. P. Stanley: \emph{BIJECTIVE PROOF PROBLEMS}, 2009, \url{http://www-math.mit.edu/~rstan/bij.pdf}, Problem 28.

\bibitem{STANLEY_AC} R. P. Stanley: \emph{TOPICS IN ALGEBRAIC COMBINATORICS}, 2011, \url{http://www-math.mit.edu/~rstan/algcomb/algcomb.pdf}
\end{thebibliography}
\end{document}